\documentclass{amsart}[12 pt]
\usepackage{amsmath, amsfonts, amssymb, amsthm, euscript, amscd, latexsym, mathrsfs, bm, color, bbm, xcolor,mathtools, varwidth, multicol, xfrac, nicefrac}
\usepackage[belowskip=-70pt,aboveskip=0pt]{caption}
 \usepackage{titlesec}
 \usepackage{enumitem}
\usepackage{lipsum}
 \usepackage[pdftex,colorlinks=true,
                       pdfstartview=FitV,
                       linkcolor=blue,
                       citecolor=blue,
                       urlcolor=blue,
           ]{hyperref}

\usepackage{overpic}
\usepackage{epsfig}
\usepackage{graphicx}
\usepackage{graphics}

\def\aa#1{ \begin{align*} #1 \end{align*} }
\def\aaa#1{ \begin{align} #1 \end{align} }
\def\mm#1{ \begin{multline*} #1 \end{multline*} }
\def\mmm#1{ \begin{multline} #1 \end{multline} }

\def\begeq{\begin{equation} \begin{cases}} 
\def\endeq{ \end{cases} \end{equation}}
\def\eq#1{ \begeq #1 \endeq }
\def\bege{\begin{equation*} \begin{cases}} 
\def\ende{ \end{cases} \end{equation*}}
\def\eqq#1{ \bege #1 \ende}
\def\eqn#1{ \begin{equation} \begin{aligned} #1 \end{aligned}\end{equation}}

\newtheorem{thm}{\sc Theorem}
\newtheorem{lem}{\sc Lemma}

\newtheorem{rem}{\sc Remark}

\newtheorem*{thm*}{\sc Theorem}

\newcommand{\pl}{\partial}
\newcommand{\gt}{\geqslant}
\newcommand{\lt}{\leqslant}

\newcommand{\te}{\theta}
\newcommand{\sub}{\subset}
\newcommand{\dl}{\delta}
\newcommand{\al}{\alpha}
\newcommand{\gm}{\gamma}
 \newcommand{\Gm}{\Gamma}
 \newcommand{\Dl}{\Delta}
 \newcommand{\la}{\lambda}
 
 \newcommand{\sg}{\sigma}

\newcommand{\mc}{\mathcal}

\newcommand{\C}{{\rm C}}

\newcommand{\td}{\tilde}

\newcommand{\x}{\times}
\newcommand{\mto}{\mapsto}

\newcommand{\rf}{\eqref}

\newcommand{\bi}{\begin{itemize}}
\newcommand{\ei}{\end{itemize}}
\newcommand{\be}{\begin{enumerate}}
\newcommand{\ee}{\end{enumerate}}
\newcommand{\biz}{\begin{itemize}[leftmargin=*]}
\newcommand{\bc}{\begin{center}}
\newcommand{\ec}{\end{center}}

\newcommand{\ovl}{\overline}
\newcommand{\F}{{\mathbb F}}
\newcommand{\G}{{\mathbb G}}

\DeclareMathOperator{\ind}{\mathbbm{1}}
\newcommand{\lap}{\Delta}

\newcommand{\lb}{\label}
\newcommand{\fdot}{\,\cdot\,}

\usepackage{tikz}

\newcommand{\tet}{\vartheta}

\def\Rnu{{\mathbb R}}

\def\Nnu{{\mathbb N}}

\def\ffi{\varphi}

\def\com#1{}

\long\def\symbolfootnote[#1]#2{\begingroup%
\def\thefootnote{\fnsymbol{footnote}}\footnote[#1]{#2}\endgroup}
\titleformat{\section}[hang]{\large\bfseries}{\thesection.}{1ex}{}{}
\titleformat{\subsection}[hang]{\normalsize\bfseries}{\thesubsection}{2ex}{}{}
\titleformat{\subsubsection}[hang]{\small\bfseries}{\thesubsubsection}{2ex}{}{}

\include{srctex.sty}

\allowdisplaybreaks

\begin{document}

\title[Solutions with positive components to quasilinear parabolic systems]
{Solutions with positive components \\ to quasilinear parabolic systems}

\author{Evelina Shamarova}
\address{Departamento de Matem\'atica, Universidade Federal da Para\'iba, Jo\~ao Pessoa, Brazil}
\email{evelina.shamarova@academico.ufpb.br}

\maketitle

\vspace{-10mm}

\begin{abstract}
We obtain sufficient conditions for the existence and uniqueness of solutions 
with non-negative components
to general quasilinear parabolic problems 
\eq{
\lb{P}
 \pl_t u^k =  \sum_{i,j=1}^n a_{ij} (t,x,u)\pl^2_{x_i x_j}\!u^k 
+ \sum_{i=1}^n b_i (t,x,u, \pl_x u)\pl_{x_i}\! u^k
+\, c^k(t,x,u,\pl_x u), \\
u^k(0,x) = \ffi^k(x), \\
u^k(t,\fdot) = 0, \quad \text{on } \pl \F,\\
 k=1,2, \dots, m, \quad x\in\F, \;\; t>0.
\tag{$P$}
}
Here, $\F$ is either a bounded domain or $\Rnu^n$; in the 
latter case, we disregard the boundary condition. We apply our results
to study the existence and asymptotic behavior of componentwise non-negative solutions to  
the Lotka-\!Volterra competition model with diffusion. In particular,
we show the convergence, as $t\to+\infty$, of the solution for a 2-species 
Lotka-\!Volterra model, whose coefficients vary in space and  time, 
 to a solution of the associated elliptic problem.

\end{abstract}

\vspace{3mm}

{\footnotesize

\vspace{1mm}


}

\section{Introduction}
Second-order parabolic systems arise in many branches of mathematics
and mathematical physics, so they have received a lot of attention in the literature
\cite{amann1990, lady, paoruan2007,paoruan2010, paoruan2013, SP, wang-wang, zhang}.
The main concerns observed in these works include 
the global existence of solution,  the non-negativity of the solution components,
and the asymptotic  behavior of the solution.


\subsection{Motivation}

In many parabolic problems arising in biology, physics, and chemistry, such as 
convection-reaction-diffusion systems,
 only solutions with non-negative components can be physically of interest.  
 Such systems may describe population densities, pressure or concentrations of nutrients and chemicals. Models that do not guarantee the non-negativity of the solution components may not be valid models or have limitations.

One of the most important and well-studied models in this context is 
 the Lotka-\!Volterra competition model with diffusion which arises 
 in mathematical biology and population dynamics.

 A two species Lotka-\!Volterra model, currently regarded as a standard
 model,  has received a significant consideration
in the literature since the 80s up to recent years 
\cite{cosner1984,he2016-2, he2016-1, he2017, Korman1987, Lou2019, Pao2015}. 
The parabolic system in the basis of this model reads
\eq{
\lb{lv1111}
\pl_t u = d_1 \lap_x u + u(\beta- \gm u - \dl v),\\
\pl_t v = d_2 \lap_x v + v(\rho - \sg u - \te v),
\tag{$LV$}
}
where $d_1,d_2>0$ are constants and $\beta,\gm,\dl,\sg,\te$ and $\rho$ are functions 
$[0,+\infty) \x \F \to [0,+\infty)$.  The above system is usually complemented with Dirichlet, Neumann, or Navier boundary conditions.

The $m$-species model with $m\gt 2$ was also considered by some authors
\cite{paoruan2017}.

\subsection{Statements of the main results}
In this work, we obtain  sufficient conditions for the existence and uniqueness of solutions,
to problem \rf{P}, with non-negative components, as well as sufficient conditions 
for the components of the existing
solution to \rf{P}  remain non-negative over time. 
We consider problem \rf{P} as an initial-boundary
value problem with the Dirichlet boundary condition, as well as a Cauchy problem.
To the best of our knowledge, the question of the non-negativity of solution
 components for such a general class of parabolic systems, as \rf{P}, is addressed
 for the first time. We state our main results.
\begin{thm}
\lb{l1111}
Let $\G\sub\Rnu^n$ be an open domain and
$u(t,x)$ be a $\C^{1,2}_b([0,T]\x\ovl \G)$-solution to the equation in 
\rf{P}. Assume there exists a constant $\kappa>0$ such that
for all $(t,x)\in [0,T]\x \G$,
$A(t,x,u) \gt \kappa\, I$, 
where $I\in\Rnu^{n\x n}$ is the identity matrix and $A=\{a_{ij}\}_{i,j=1}^n$.
Further assume that the coefficients of the equation in \rf{P} 
and the derivatives $\pl_x a_{ij}$ and $\pl_u a_{ij}$
are continuous in all arguments and that
 for some $k\in \{1,\ldots, m\}$,
\bi
 \item[(i)] $u^k(0,x) \gt 0$ on $\G$;
 \item[(ii)] $c^k(t,x,v,p)|_{v^k=0} \gt 0$ for $(x,t,v,p)\in [0,T]\x\ovl \G\x \Rnu^n\x \Rnu^{m\x n}$;
 \item[(iii)] $\pl_{u_k}\! c^k$ exists and is bounded on 
 \begin{center}
 $\mc R(\G, C_1,C_2):= [0,T]\x \G \x \{|u|\lt C_1\}\x \{|p|\lt C_2\}$
 \end{center}
 for some constants $C_1,C_2>0$.  
 \ei
 Then, $u^k(t,x) \gt 0$ on $[0,T]\x \G$.
\end{thm}
In the next theorem, we assume conditions
that guarantee the existence and uniqueness
of solution to the initial-boundary value problem \rf{P} (see \cite{SP}).
These conditions are listed in Subsection \ref{notation-assumptions} 
 as \ref{A1}--\ref{A6}. However, the main assumption that guarantees
the non-negativity of the solution components reads as follows:
\be
[label=($A_\arabic*$)]
\setcounter{enumi}{6}
\item \lb{A7} 
For all $k\in \{1,\ldots, m\}$,
\bi
\item[(a)]  $\ffi^k_0\gt 0 \;\; \text{on} \; \F$;
\item[(b)] $c^k(t,x,u,p)|_{u^k=0} \gt 0$ on regions $[0,T]\x \F \x \{u^i\gt 0, i\ne k\}\x \Rnu^{m\x n}.$
\ei
\ee
\begin{thm}
\lb{th1111}
Assume \ref{A1}--\ref{A6} (see Subsection \ref{notation-assumptions}).
Further assume that \ref{A7} holds and that $\pl_u c$
is bounded on regions $\mc R(\F,C_1,C_2)$.
Then, there exists a 
$\C^{1+\frac{\al}2,2+\al}([0,T]\x \F)$-solution to problem \rf{P}, 
$\al\in (0,1)$, with non-negative components, 
which is unique in the class $\C^{1,2}([0,T]\x \F)$.
\end{thm}
\begin{rem}
\rm
The parabolic H\"older space $\C^{1+\frac{\al}2,2+\al}([0,T]\x \F)$
is defined in Subsecion \ref{notation-assumptions}.
\end{rem}
If we assume that the coefficient $c=(c^1,\ldots, c^n)$ does not 
depend on $\pl_x u$,
as it is in the case of the Lotka-\!Volterra system \rf{lv1111}, we can use 
the weaker assumtion \ref{A2'} (below) instead of \ref{A2}. 
This is important because 
\ref{A2} does not hold for system \rf{lv1111}.
\be
[label=($A'_\arabic*$)]
\setcounter{enumi}{1}
\item \lb{A2'} 
There exist constants $d_1,d_2>0$ such that 
\aa{
 &(c(t,x,u,p),u) \lt d_1 + d_2|u|^2 \quad \text{for\;} u\in \{u_i\gt 0, i=1,\ldots,m\},\\
 &(t,x,p)\in [0,T]\x \F\x \Rnu^{m\x n}.
 }
\ee
Remark that in \ref{A2}, the above inequality holds for all $u\in\Rnu^m$, 
while in \ref{A2'}, it holds only for $u$ with non-negative components. 

Our results on the non-negativity of the solutions components
 are directly applicable to the $m$-species Lotka-\!Volterra competition 
 model with diffusion.
  Although we mostly concentrate on the $2$-species Lotka-\!Volterra model,
 some of our results are valid for  the $m$-species model with $m\gt 2$.
 Namely, for  the $m$-species model,
 we obtain the existence and uniqueness of solution
 with non-negative  components along with sufficient conditions 
 for extinction of one of the species as $t\to+\infty$. 
 For the $2$-species model, we additionally
 prove the convergence of solution to the solution of the associated elliptic system.
 Namely, we obtain the following result.
\begin{thm}
\lb{th-lv2222}
Let $(u,v)$ be 
a componentwise non-negative bounded $\C^{1+\frac\al2,2+\al}$-solution to \rf{lv1111}
on $[0,+\infty)\x \F$.
Assume that the coefficients $\beta, \gm, \dl, \rho, \sg, \te$ of \rf{lv1111} 
possess second-order derivatives in $x$ of class $\C^{\frac\al2,\al}$\!.
Further assume that the above coeffcients are differentiable in $t\in [0,+\infty)$  and 
 $\pl_t \beta\gt 0$, $\pl_t \gm  \lt 0$, $\pl_t \dl\lt 0$, $\pl_t \rho\lt 0$,
$\pl_t \sg\gt 0$, $\pl_t \te\gt 0$ on $[0,+\infty)\x \F$;
$\pl_t u(0,x)\gt 0$, $\pl_t v(0,x)\lt 0$ on $\F$.
Finally, assume that the pointwise limits 
$\bar \beta$, $\bar \gm$, $\bar \dl$, $\bar \rho$, $\bar \sg$, $\bar \te$
of $\beta, \gm, \dl, \rho, \sg, \te$  as $t\to+\infty$ (respectively)
are of class $L^1_{loc}$ in all arguments.
Then, there exists a pair of functions $(\bar u(x),\bar v(x))$ such that
for all $x\in\F$,
\aaa{
\lb{limits}
\lim_{t\to+\infty} u(t,x) = \bar u(x) \quad \text{and} \quad
\lim_{t\to+\infty} v(t,x) = \bar v(x).
}
Moreover,  $(\bar u(x),\bar v(x))$ is a weak solution  
to the elliptic system 
\eq{
\lb{el-1111}
d_1 \lap_x \bar u + \bar u(\bar \beta- \bar \gm\bar u - \bar \dl\bar v) = 0,\\
d_2 \lap_x \bar v + \bar v(\bar \rho - \bar \sg\bar u - \bar \te\bar v) = 0.
}
\end{thm}
\begin{rem}
\lb{weak1111}
\rm
In Theorem \ref{th-lv2222}, by a {\it weak solution} to system \rf{el-1111}, we understand
a pair $(\bar u,\bar v)$ such that for any pair of
test functions $\eta(x)$ and $\zeta(x)$, supported on $\F$,
it holds that
\eq{
\lb{weak1111}
 \int_\F \big(d_1 \lap_x \eta\, \bar u  
+ \eta\, \bar u (\bar \beta- \bar \gm\bar u - \bar \dl\bar v) \big) dx = 0, \\
\int_\F \big(
d_2 \lap_x \zeta \,\bar v + \zeta \,
\bar v(\bar \rho - \bar \sg\bar u - \bar \te\bar v)\big) dx = 0. 
}
\end{rem}


\subsection{Our methods and comparison with other works}
As the first step, we prove Theorem \ref{l1111}, which is
 a sufficient condition for a certain component of
an existing classical solution to \rf{P} to remain non-negative. To this end, we use
integration by parts in combination with some estimates. The
 existence and uniqueness of solution with non-negative
components is a consequence of results of \cite{SP} (see also \cite{lady})
and the aforementioned sufficient condition for the non-negativity of one fixed component. 
These results are applied to show the existence and uniqueness of
solution for the general $m$-species Lotka-\!Volterra model and for
studying the asymptotic behavior of the $2$-species model.
More specifically, under the conditions introduced in Theorem \ref{th-lv2222},
 the convergence \rf{limits} follows from the non-negativity or non-positivity
of the time derivatives of the solution components. The non-negativity (non-positivity)
of time derivatives follows, in turn, from Theorems \ref{l1111} and \ref{th2222}.
In contrast with 
other works \cite{cosner1984, Korman1987, Pao2015, Pao2015-1, paoruan2017},
our method  works for the coefficients $\beta,\gm,\dl, \rho, \sg,\te$ which 
can vary in space and time. Most of models proposed in the literature
have significantly stronger restrictions on these coefficients. More specifically,
in most of papers, e.g., \cite{Korman1987, Pao2015, Pao2015-1, paoruan2017},  the coefficients 
are constant; other authors consider some of the coefficients varying only in space
and some being constant. 

On the other hand, the questions of existence, uniquenes, and asymptotic behavior of 
solutions with non-negative components have been treated in the literature (see,
e.g., \cite{cosner1984, Korman1987, Pao2015, Pao2015-1, paoruan2017})
by totally different methods. Most of the aforementioned works employ
the sub- and supersolution method along with comparison theorems. 
We mention here
papers \cite{he2016-2, he2016-1, he2017} 
that contain a detailed study of the asymptotic behavior of solutions
 for constant or varying in space coefficients.
 
 To summarize, we mention once again that in this work, we not only obtain
 the existence and uniqueness of solutions with non-negative components
 for a large class of parabolic systems of the form \rf{P}, 
 but also, we show how  these results can be applied to study the Lotka-\!Volterra system
 with diffusion by the methods that differ from those used in the literature
 and that work for the coefficients of the model that vary in space and time.


\subsection{Structure of this work}
In Subsection \ref{s21}, we formulate Assumptions \ref{A1}--\ref{A6} which
are used (sometimes complemented with Assumptions \ref{A7} and \ref{A2'})
to obtain the non-negativity of solution components.
In Subsection \ref{s22}, we consider the initial-boundary value problem \rf{P} and
prove Theorems \ref{l1111} and \ref{th1111}. In the same section,
we obtain the result of Theorem \ref{th1111} in the situation when 
the coefficient $c$ does not depend on the gradient of the solution.
This allows to use the weaker Assumption \ref{A2'} instead of \ref{A2}.
In Subsection \ref{s23},  we obtain a result on the non-negativity of
solution components for problem \rf{P} regarded as a Cauchy problem.
Finally, in Section \ref{s-lv-1111}, we apply the results of Subsection
\ref{s22} to the Lotka-\!Volterra competition model with diffusion. In particular,
in this section, we prove Theorem \ref{th-lv2222}.

\section{Results on the non-negativity of the solution components for Problem \rf{P}}

\subsection{Notation and assumptions}
\lb{s21}
\lb{notation-assumptions}
Everywhere below, $\F\sub\Rnu^n$ denotes 
an open bounded domain with a piecewise-smooth boundary 
$\pl\F$ and non-zero interior angles;
$\F_T =[0,T] \x \F$. For two positive constants $C_1$ and $C_2$,
define
\aa{
&\mc R(\F, C_1,C_2):= \F_T \x \{|u|\lt C_1\}\x \{|p|\lt C_2\},\\
&\mc R_1(\F, C_1):= \F_T \x \{|u|\lt C_1\}.
}
The H\"older space $\C^{2+\al}(\ovl \F)$,
 $\al\in (0,1)$,
is understood as the Banach space with the norm
\aaa{
\lb{ho1}
\|\phi\|_{\C^{2 +\al}(\ovl \F)} = \|\phi\|_{\C^2(\ovl \F)} + [\phi'']_\al,
\quad \text{where} \quad
[\td \phi]_\al = \sup_{x,y\in \ovl \F,\, 0<|x-y|<1}\frac{|\td\phi(x)- \td\phi(y)|}{|x-y|^\al}.
}
For a function $\ffi(x,t)$, the H\"older constant 
with respect to $x$ is defined as 
\aa{
[\ffi]^x_{\al} =  \sup_{x,y\in \ovl \F,\, 0<|x-y|<1} \frac{|\ffi(x,t) - \ffi(y,t)|}{|x-y|^\al},
}
i.e., it is understood as a function of $t$; $[\ffi]^t_{\al}$ is defined likewise.
The H\"older space $\C^{\frac{\al}2,\al}(\ovl \F_T)$ is 
the space of functions $u(t,x)$ possessing the finite norm
\aa{
\|u\|_{\C^{\frac\al2,\al}(\ovl \F_T)} =  \|u\|_{\C(\ovl \F_T)}+
\sup_{t\in [0,T]}[u]_{\al}^x + \sup_{x\in \ovl \F }[u]_{\frac\al2}^t.
}
Furthermore, the parabolic H\"older space $\C^{1+\frac{\al}2,2+\al}(\ovl \F_T)$, $\al\in (0,1)$,
is defined as the Banach space of functions $u(t,x)$ with the norm
\mmm{
\lb{norm1111}
\|u\|_{\C^{1+\frac{\al}2,2+\al}(\ovl \F_T)} =
\|u\|_{\C^{1,2}(\ovl \F_T)} + \sup_{t\in [0,T]}[\pl_t u]_{\al}^x + \sup_{t\in [0,T]}[\pl^2_{xx} u]_{\al}^x \\
+ \sup_{x\in \ovl \F}[\pl_t u]_{\frac{\al}2}^t + \sup_{x\in \ovl \F}[\pl_x u]_{\frac{1+\al}2}^t
+ \sup_{x\in \ovl \F}[\pl^2_{xx} u]_{\frac{\al}2}^t.
} 
Finally, the parabolic H\"older space $\C_b^{1+\frac{\al}2,2+\al}([0,T]\x \Rnu^n)$
is defined similar to $\C^{1+\frac{\al}2,2+\al}(\ovl \F_T)$, but $\F$ is replaced
by $\Rnu^n$ and the first term in \rf{norm1111} is replaced by
$\|u\|_{\C^{1,2}_b([0,T] \x \Rnu^n)}$, where 
$\C_b^{1,2}([0,T]\x \Rnu^n)$ denotes the space of bounded continuous functions
whose mixed derivatives up to the second order in $x\in\Rnu^n$ and first order 
in $t\in [0,T]$ are bounded and continuous on $[0,T]\x \Rnu^n$.

In the assumptions below,  we let the functions $\mu(s)$ and $\hat \mu(s)$,
be non-decreasing and, respectively, non-increasing,
continuous,  defined for positive arguments, and
taking positive values.
Furthermore, the functions $\tet_1(s)$ and $\tet_2(s)$ 
are assumed continuous,  defined for positive arguments,
taking positive values, and non-decreasing with respect to each argument, 
whenever the other arguments are fixed.
Throughout this section, we assume
 \be
[label=($A_\arabic*$)]
\item  \lb{A1} The PDE in \rf{P} is uniformly parabolic on 
$\F_T\x \Rnu^m$, i.e., 
\aa{
\hat \mu(u) I \lt A(t,x,u) \lt \mu(u) I,
}
where $I\in\Rnu^{n\x n}$ is the identity matrix and $A=\{a_{ij}\}_{i,j=1}^n$.
\item \lb{A2} 
There exist constants $d_1,d_2>0$ such that 
\aa{
 &(c(t,x,u,p),u) \lt d_1 + d_2|u|^2\;\; \text{for all } (t,x,u,p)\in 
 \F_T\x\Rnu^m\x \Rnu^{m\x n},\\
 &\text{where}\;\;  c(t,x,u,p)  = (c^1,\ldots, c^n)(t,x,u,p).
 }
\item \lb{A3}  $\ffi\in \C^{2+\al}(\F)$, $\al\in (0,1)$.
\item \lb{A4}  For all  $(t,x,u,p)\in \mc R_1(\F,C_1)\x\Rnu^{m\x n}$, for each $i=1,\ldots n$,
\aa{
&{\rm (a)\quad} |b_i(t,x,u,p)|  \lt \tet_1(|u|)(1+|p|); \\
&{\rm (b)\quad} |c(t,x,u,p)|  \lt \tet_2(|u|, |p|)(1+|p|)^2, \qquad
\lim_{|p|\to \infty} \tet_2(|u|,|p|) = 0.
}
 \item \lb{A5}  
 On regions $\mc R_1(\F_T,C_1)$, 
 the coefficients $a_{ij}$ of \rf{P} 
 possess  continuous second-order derivatives.
 On regions $\mc R(\F_T,C_1,C_2)$,
the coefficients $b_i$ and $c$ possess
 first-order continuous partial derivatives in all arguments (except $x$),
 which are $\al$-H\"older continuous in $x$, $\al\in (0,1)$, 
 and have bounded H\"older constants.
 \item \lb{A6} The following compatibility condition 
  holds for $x\in \pl \F$: $\ffi(x) = 0$ and
\aa{ 
 \sum_{i,j=1}^n  a_{ij} (0,x,0)\pl^2_{x_i x_j} \ffi(x) +
\sum_{i=1}^n   b_i (0,x,0,\pl_x \ffi(x)) \pl_{x_i} \ffi(x)
 + c(0,x, 0, \pl_x \ffi(x)) =0. 
  }
  \ee
\subsection{Initial boundary-value problem}
Here, we prove Theorems \ref{l1111} and \ref{th1111}.
\lb{s22}
\begin{proof}[Proof of Theorem \ref{l1111}]
Define $u^k_{+} = \max\{u^k,0\}$ and $u^k_{-}= \min\{u^k,0\}$, so that
 $u^k = u^k_{+} + u^k_{-}$. It is sufficient to show that $u^k_{-}(t,x) = 0$.

We agree that throughout the proof $\gm_i$, $i=1,2,\ldots$, are positive constants.
Take an arbitrary open ball $B$ whose closure is in $\G$. Let $\eta: \ovl B\to [0,1]$ 
be a smooth function, taking positive values in $B$ and zero on $\pl B$.
For a fixed $t\in [0,T]$, we multiply the equation in \rf{P} by $\eta^2(x) u^k_{-}(t,x)$ and 
integrate over $B$. By the integration-by-parts  formula (see e.g \cite{lady}, p.60), we obtain
\mmm{
\lb{pde-int}
\int_B\eta^2 \pl_t u^k u^k_- dx  
 +\sum_{i,j=1}^n \int_B \eta^2 a_{ij} (t,x,u) \pl_{x_j} u^k_- \pl_{x_i}u^k_- dx \\ =
- \sum_{i,j=1}^n \int_B 
(\eta^2 \pl_{x_i}a_{ij}(t,x,u(t,x))+2\eta \,a_{ij}(t,x,u)
\pl_{x_i}\eta) \pl_{x_j} u^k_- u^k_-  dx \\
+ \sum_{i=1}^n \int_B\eta^2 b_i (t,x,u, \pl_x u)\pl_{x_i} u^k_- u^k_- dx  
+ \int_B \eta^2 c^k(t,x,u,  \pl_x u) u^k_- dx.
}
It holds that $\pl_{x_j} u^k_- = \pl_{x_j} u^k \ind_{\{u^k<0\}}$ and  $\pl_{x_j} u^k_+ = \pl_{x_j} u^k \ind_{\{u^k>0\}}$ in the sense of weak derivatives 
(see \cite[Lemma 7.6]{gilbarg}). 
Furthermore, we show that  $\pl_t u^k_- = \pl_t u^k$ and $\pl_t u^k_+=0$ on the set 
$G_- = \{(t,x) \in\F_T: u^k(t,x)<0\}$. 
Indeed,  since $u^k(t,x) = u^k_-(t,x)$ on $G_-$,
by continuity, $u^k(t+\Dl t, x)=u^k_-(t+\Dl t,x)$ for sufficiently small $\Dl t$, depending on $(t,x)$. 
Hence, $\pl_t u^k_- = \pl_t u^k$ on $G_{-}$. 
For the same reason, $u^k_+(t+\Dl t,x) = u^k_+(t,x) = 0$ on
$G_-$  for sufficiently small $\Dl t$. 
Therefore, $\int_B \!\eta^2 \pl_t u^k u^k_- dx = \int_B\!\eta^2 \pl_t u^k_- u^k_- dx  
=\frac12 \pl_t \!\int_B\!\eta^2 (u^k_-)^2 dx$. By the strong ellipticity
of the second-order differential operator
involved in \rf{P}, the left-hand side of \rf{pde-int} can be evaluated from below by
\aa{
 \frac12\, \pl_t \int_B \eta^2 (u^k_-)^2 dx + \gm_1 \int_B \eta^2 |\pl_x u^k_-|^2  dx.
}
Further, by boundedness of $b_i$, $\pl_x a_{ij}$, $\pl_u a_{ij}$, and $\pl_x u$, the two 
terms on the right-hand side of \rf{pde-int}, containing $\pl_{x_i} u^k_- u^k_-$,
can be evaluated  from above by
\aa{
\gm_3\int_B \eta \, |\pl_x u^k_-| |u^k_-| dx \lt \gm_1\int_B |\pl_x u^k_-|^2 dx
+ \gm_4 \int_B \eta^2 (u^k_-)^2 dx.
}
Finally, note that
\aa{
c^k(t,x,u,\pl_x u) = c^k(t,x,u,\pl_x u)|_{u^k=0} + u^k \int_0^1 \pl_{u_k}
c^k(t,x,u_1,\dots, \la u_k, \dots u_n,\pl_x u) d\la.
} 
Since 
$c^k(t,x,u,\pl_x u)|_{u^k=0}\gt 0$  by assumption,
the last term on the right-hand side of \rf{pde-int} can be evaluated from above by
\aa{
\int_B \eta^2 c^k(t,x,u,\pl_x u)|_{u^k=0}\, u^k_- dx + \gm_5 \int_B  \eta^2 (u^k_-)^2 dx
 \lt \gm_5 \int_B \eta^2 (u^k_-)^2 dx.
}
By the last three estimates and since $\eta^2(x) \lt 1$, 
there exists a constant $\gm_0>0$ such that 
\aa{
\pl_t \int_B \eta^2 ( u^k_-)^2 dx  \lt \gm_0  \int_B  \eta^2 (u^k_-)^2 dx.
}
Since $(\ffi^k)_- = 0$,  by Gronwall's inequality, $u^k_- = 0$ on $B\x [0,T]$, 
and hence on $\F_T$. 
\end{proof}
In Theorems \ref{th1111} and \ref{th2222} below, 
we obtain the existence and uniqueness of solutions with non-negative components
under \ref{A1}--\ref{A6} (assumptions that guarantee the existence and uniqueness result, see \cite{lady, SP}) 
along with additional assumptions on the coefficient $c(\fdot)$.
Theorem \ref{th2222} provides the aforementioned result
under the weaker assumption \ref{A2'}, instead of \ref{A2};
however, it is assumed that the coefficient $c$ does not depend
on $\pl_x u$. 
Theorem \ref{th2222} will be used for applications to the Lotka-\!Volterra
competiton model with diffusion in Section \ref{s-lv-1111}.
\begin{proof}[Proof of Theorem \ref{th1111}]
Note that the existence of a unique $\C^{1,2}(\F_T)$-solution $u(t,x)$ to problem
\rf{P} is known under \ref{A1}--\ref{A6} (see, e.g., \cite{SP}).
We consider the problem 
\aaa{
\lb{sv2222}
 &\pl_t v^k =  \sum_{i,j=1}^n a_{ij} (t,x,u)\pl^2_{x_i x_j}v + 
 \sum_{i=1}^n b_i (t,x,u, \pl_x u)\pl_{x_i} v^k 
+ c^k(t,x,|v^1|, \ldots, |v^m|,  \pl_x u),\\ 
 &v(0,x) = \ffi(x)  \notag
}
with respect to $v=(v^1,\ldots, v^m)$,  together  with the equations
\aa{
  v^k(t,x) = \int_{\F}G(t,x,0,y) \ffi^k(y)  dy 
 +\int_0^t \int_{\F} G(t,x,s,y) c^k(s,y,|v^1|, \ldots, |v^m|,  \pl_x u) dy ds
 }
where $k=1,2,\ldots, m$ and $G(t,x,s,y)$ stands for the (real-valued) Green function 
for the operator 
$L[v^k] = \pl_t v^k - \sum_{i,j=1}^n a_{ij} (t,x,u)\pl^2_{x_i x_j}v^k - 
 \sum_{i=1}^n b_i (t,x,u, \pl_x u)\pl_{x_i} v^k$. The existence of the Green
 function for $L$ and the domain $\F$ is known due to \cite[Section IV, \S 16]{lady}.
 Since the coefficients of $L$ do not depend on $k$, we can rewrite the above
 equation in the vector form
 \aaa{
   \lb{v2222}
  v(t,x) = \int_{\F}G(t,x,0,y) \ffi(y)  dy 
 +\int_0^t \int_{\F} G(t,x,s,y) c(s,y,|v^1|, \ldots, |v^m|,  \pl_x u) dy ds.
 }
Note that 
 equation \rf{v2222} possesses a unique solution in $\C(\ovl \F_T)$.
 To show this, define the map $\Gm: \C(\ovl \F_T)\to \C(\ovl \F_T)$
 such that $\Gm(v)$ is given by the right-hand side of \rf{v2222}. 
Theorem 16.3 from \cite[Section IV, \S 16]{lady} on estimates
 for the Green function implies that $\Gm(v)$
 is H\"older continuous in $t$ and $x$ with bounded H\"older constants.
 Hence, the map $\Gm$ is well-defined.  Let $v,\td v\in \C(\ovl \F_T)$.
 Then, by the inequality $||a|-|b||\lt |a-b|$, $a,b\in\Rnu$, and since
 $\pl_u c$ is bounded, we obtain that
 \mm{
|\Gm v(t,x)- \Gm\td v(t,x)| \lt C \int_0^t 
   \sup_{z\in\F}|v(s,z)-\td v(s,z)| \int_\F p_{t-s}(x-y) dy \, ds\\
   \lt C \int_0^t  \sup_{x\in\F}|v(s,x)-\td v(s,x)| \, ds.
 }
 Above, we used the fact that  $G(t,x;s,y)$ possesses bounds by 
 Gaussian densities (see \cite[Section IV, \S 16]{lady}),
Then, $\Gm^N=\underbrace{\Gm\circ \dots \circ \Gm}_N$ is a contraction map 
$\C(\ovl \F_T)\to \C(\ovl\F_T)$ for some $N\in\Nnu$ since
\aa{
\|\Gm^N v - \Gm^N \td v\|_{\C(\ovl \F_T)} \lt \frac{C^N T^N}{N!} 
\|v -\td v\|_{\C(\ovl \F_T)}.
}
 Therefore, again by  \cite[Section IV, \S 16]{lady}, the solution $v$ to \rf{v2222}
is a $\C^{1,2}(\F_T)$-solution to problem \rf{sv2222}.
 Further, by Theorem \ref{l1111}, $v^k\gt 0$ for all $k$, which implies that
$v$ is a  $\C^{1,2}(\F_T)$-solution to 
\eqn{
\lb{v3333}
 &\pl_t v =  \sum_{i,j=1}^n a_{ij} (t,x,u)\pl^2_{x_i x_j}v + \sum_{i=1}^n b_i (t,x,u, \pl_x u)\pl_{x_i} v 
+ c(t,x,v^1, \ldots, v^m,  \pl_x u),\\ 
&v(0,x) = \ffi(x). 
}
However, $u$ is a $\C^{1,2}(\F_T)$-solution to
problem \rf{v3333}.  Moreover, under
\ref{A1}--\ref{A6}, $u$ is unique solution (see \cite{lady, SP}). This implies that $u=v$,
and hence, $u$ has non-negative components.
\end{proof}
Lemma \ref{lem2222} below is a maximum principle for problem \rf{P},
adapted for solutions with non-negative components.
The proof of Lemma \ref{lem2222} follows the lines
of \cite[Lemma 2]{SP}. However, since the result of \cite{SP} cannot be applied directly,
we repeat the proof here for the reader's convenience. 
\begin{lem}
\lb{lem2222}
 Assume \ref{A1}, \ref{A2'}, and \ref{A3}. 
 Let $u(t,x)$ be a  $\C^{1,2}(\F_T)$-solution
to problem \rf{P} whose components are non-negative.
 Then,  
 \aaa{
 \lb{m1111}
 \sup_{\ovl\F_T}|u(t,x)| \lt \max \big\{  e^{(d_2+1) T} \sup_{\ovl\F} |\ffi_0(x)|, \, 
\sqrt{d_1} \big\}.
 }
\end{lem}
\begin{proof}
Let $v(t,x)=u(t,x) e^{-\la t}$, where $\la>0$ is to be fixed later. 
Then, $v$ satisfies the equation
\aa{
 \pl_t v + \la v  - \sum_{i,j=1}^{n} a_{ij} (t,x,{u}){v}_{x_i x_j} - e^{-\la t} c(t,x,u,u_x)
 -  \sum_{i=1}^{n} b_i(t,x,u,u_x)v_{x_i} =0.
 }
 Multiplying the above identity scalarly by $v$, 
 and noticing that $(v_{x_ix_j},v) = \frac12 \pl^2_{x_i x_j} |v|^2 - (v_{x_i},v_{x_j})$,
 we obtain
\mmm{
 0=\frac12 \pl_t |v|^2 + \la |v|^2  
 - \frac12 \sum_{i,j=1}^{n}  a_{ij} (t,x,{u})\pl^2_{x_i x_j} |v|^2 -
  e^{- 2 \la t} (c(t,x,u,u_x),u)
 \\ + \sum_{i,j=1}^{n}  a_{ij} (t,x,{u})  (v_{x_i}, v_{x_j}) -
 \frac12  \sum_{i=1}^n b_i(t,x,u,u_x) \pl_{x_i} |v|^2,
  \label{mi1111}
 }
 where $u$ and $v$ are evaluated at $(t,x)$. Note that for the function $w=|v|^2$
 one of the following situations is necessarily realized: 
 1) $w$ achieves its maximum on $(0,T]\x \pl \F$; 
 2) $w$ achieves its maximum on $\{t=0\}\x \F$; 
 3) there exists 
$(t_0,x_0)\in (0,T] \x \F$ such that $w(t_0,x_0) = \sup_{\ovl\F_T}w(t,x)$. 
In case 1), the statement follows trivially. In case 2), one has
\aa{
|u(t,x)| \lt e^{\la T} |v(t,x)| \lt e^{\la T} \sup_{\bar\F} |\ffi(x)|
}
Finally, in case 3),   
\aa{
\pl_x w(t_0,x_0)=0 \quad \text{and} \quad  \pl_t w(t_0,x_0) \gt 0.
}
By \cite[Lemma 1]{SP}, the term 
$-\frac12\sum_{i,j=1}^{n}  a_{ij} (t,x,{u})\pl^2_{x_i x_j} |v|^2$
in \rf{mi1111}  is non-negative at $(t_0,x_0)$.
Further remark that at  $(t_0,x_0)$, the last term equals zero,
 the penultimate term is non-negative by \ref{A1}, 
and the first two terms are non-negative.
Finally, note that $e^{- 2 \la t} (c(t,x,u,u_x),u) \lt d_1 e^{-2\la t} + d_2 w$.
Consequently,  \rf{mi1111} implies that at point $(t_0,x_0)$ 
\aa{
 0\gt -d_1e^{-2\la t_0} - d_2w(t_0,x_0) + \la w(t_0,x_0). 
}
Picking $\la  = d_2   + 1$, we obtain that $|u(t_0,x_0)|^2 \lt d_1$, which concludes the proof.
\end{proof}
\begin{thm}
\lb{th2222}
Let the assumptions of Theorem \ref{th1111} hold,
 except the assumption \ref{A2} which is replaced with the weaker
 assumption \ref{A2'}.
Assume that the coefficient $c(\fdot)$ in \rf{P}
does not depend on $\pl_x u$ and that $\pl_u c$ is bounded
on regions $\mc R(\F,C_1)$ for any positive constant $C_1$. 
Then, there exists a 
$\C^{1+\frac{\al}2,2+\al}(\F_T)$-solution to problem \rf{P} with non-negative 
components which is unique in the class $\C^{1,2}(\F_T)$.
\end{thm}
\begin{proof}
First, we note that the conclusion of the theorem holds under the assumption
that $\pl_u c$ is globally bounded since the latter condition implies \ref{A2}. 
Indeed, by boundedness of $\pl_u c$, 
$|c(t,x,u)| \lt  \gm_1 + \gm_2 |u|$ for some positive constant $\gm_1$ and $\gm_2$,
which immediately implies \ref{A2}. Hence, if $\pl_u c$ is globally bounded,
 the existence of a unique 
$\C^{1,2}(\F_T)$-solution $u(t,x)$ to problem \rf{P} follows from Theorem \ref{th1111}. 
Moreover,  the components of $u$ are non-negative. 
 
 Let the constant $M$ be defined by the right-hand side of \rf{m1111}.
 By \ref{A2'} and Lemma \ref{lem2222},
 any solution $u$ to \rf{P} with non-negative components is bounded by $M$.
 In problem \rf{P},
 substitute $c(t,x,u)$ with $\td c(t,x,u): = c(t,x, \xi_{M}(u) u)$,
 where $\xi_M$ is a mollification of the indicator function
  $\ind_{\{|u|\lt 2M\}}$ with the property that
 $\xi_M(u) = 1$ if $|u|\lt M$.   
 We will refer to the corresponding problem as $(\td P)$. 
 By what was proved, problem $(\td P)$ has a unique $\C^{1,2}(\F_T)$-solution
 $\td u$; moreover, the components of $\td u$ are non-negative. 
 By Lemma \ref{lem2222}, $|\td u(t,x)|\lt M$. 
 Hence, $\td u$ is a also a solution to the original problem \rf{P}. 
 Finally, we note that under the \ref{A1}, \ref{A3}--\ref{A6}, the solution
 to \rf{P} is unique. The uniqueness is obtained 
 in \cite[Theorem 9, Section 2.7]{SP}. Remark that the aforementioned 
 result on uniqueness (exposed in \cite{SP}) does not require 
 \ref{A2}.
 \end{proof}

\subsection{Cauchy problem}
\lb{s23}
\begin{thm}
\lb{th3333}
Assume that \ref{A1}--\ref{A7} are fulfilled with $\F = B_r(0)$, that is,
for centered balls $B_r(0)$ of all radii $r>0$.
Further assume that $\pl_u c$ is bounded on regions
$\mc R(B_r(0),C_1,C_2)$ for all $r>0$. 
Then, there exists a 
$\C^{1+\frac{\al}2,2+\al}_b([0,T] \x \Rnu^n)$-solution with 
non-negative components  to problem \rf{P},  which is unique in the class $\C^{1,2}_b([0,T] \x \Rnu^n)$.
\end{thm}
\begin{proof}
Let 
$\zeta_r(x)$ be a smooth cutting function for the ball $B_r(0)$, $r>1$; 
namely, $\zeta_r(x) =1$ if $x\in B_{r-1}(0)$,
$\zeta_r(x) = 0$ if $x\notin B_r(0)$, 
$\zeta_r$ has bounded derivatives of all orders that do not depend on $r$.
One can take $\zeta_r$ to be the mollified indicator function 
of the ball $B_{r-1/2}(0)$. Consider the PDE 
\aaa{
\label{PDE1}
 \pl_t u =  \sum_{i,j=1}^n a_{ij} (t,x,u)\pl^2_{x_i x_j}u + \sum_{i=1}^n b_i (t,x,u, \pl_x u)\pl_{x_i} u 
+ \xi_r(x) c(t,x,u,  \pl_x u) 
}
 with the boundary function 
\eq{
\lb{bc2222}
\psi(t,x) =
\ffi(x) \xi_r(x), \quad x\in \{t=0\} \x B_r, \\
0, \quad (t,x) \in  [0,T] \x \pl B_r.
}
By Theorem \ref{th1111}, there exists a unique 
$\C^{1,2}([0,T], \ovl B_r)$-solution $u_r$ to problem (\ref{PDE1}-\ref{bc2222}).
Moreover, $u^k_r\gt 0$ for all $k=1,\ldots, m$.

It is known that the unique $\C^{1,2}_b([0,T]\x\Rnu^n)$-solution $u(t,x)$ to problem \rf{P}  can constructed by the diagonalization argument
(see \cite{lady} or  \cite[Theorem 10]{SP}). Namely, there is a sequence of solutions $u_{r_l}(t,x)$ to initial-boundary
value problems (\ref{PDE1}-\ref{bc2222})  with $r=r_l$, where
$\lim_{l\to+\infty} r_l = +\infty$,  such that for each $x\in\Rnu^n$,
\aa{
u(t,x) = \lim_{l\to\infty} u_{r_l}(t,x).
}
By Theorem \ref{l1111}, $u^k_{r_l}(t,x)\gt 0$ on $[0,T] \x B_{r_l}$. 
Therefore, $u^k(t,x)\gt 0$ on $[0,T]\x\Rnu^n$.
\end{proof}
\begin{thm}
\lb{th4444}
Assume that the conditions of Theorem \ref{th3333} hold,
except \ref{A2}, which is replaced with \ref{A2'} and fulfilled
on regions $\mc R(B_r(0),C_1,C_2)$ for all $r>0$ and
$C_1,C_2$ depending on $r$.
Next, we assume that the coefficient $c$ in \rf{P}
does not depend on the last argument $p$ and that $\pl_u c$ is bounded
on regions $\mc R(B_r(0),C_1)$ for and all $r>0$ and for some positive
constant $C_1$, depending on $r$. 
Then, there exists a 
$\C^{1+\frac{\al}2,2+\al}_b([0,T] \x \Rnu^n)$-solution with non-negative 
components  to problem \rf{P},
which is unique in the class $\C^{1,2}_b([0,T] \x \Rnu^n)$.
\end{thm}
\begin{proof}
The proof is almost the same as of Theorem \ref{th3333}, except for the fact
that the solutions $u_r$ are introduced by virtue of Theorem \ref{th2222}. 
\end{proof}

\section{Applications: Lotka-\!Volterra competition model with diffusion}
\lb{s-lv-1111}
\subsection{Existence and asymptotic behavior for the
$m$-species Lotka-\!Volterra model}
Here, we consider the following Lotka-\!Volterra competition model, 
widely studied in the literature, mostly in the dimension $n=2$: 
\eq{
\lb{lv-nd}
\pl_t u^k = d_k \lap_x u^k + u^k(\beta_k - \gm_{kk} u^k - 
\sum_{i\ne k}^m \gm_{ki} u^i),  \\
u^k(0,x) = \phi^k(x), \quad x\in \F,\\
u^k(t,\fdot) = 0 \;\; \text{on} \;\; \pl\F;\quad
 k=1,\ldots, m.
}
Above, $\beta_k$, $\gm_{ki}$, $i,k = 1,\ldots, m$ are non-negative functions
on $[0,+\infty)\x \F$, $d_k>0$ are constants. 
We start by applying the results of the previous section 
to show the so-called coexistence of positive states
for problem \rf{lv-nd}. Namely, we have the following result.
\begin{thm}
\lb{th-lv1111}
Let
$\beta_k$, $\gm_{ki}$, $i,k = 1,\ldots, m$,  be functions $[0,+\infty) \x \F \to [0,+\infty)$
continuously differentiable in $t$, $\al$-H\"older continuous in $x$
$(0<\al<1)$ and
such that the time derivatives and H\"older constants are bounded
on regions $[0,T] \x \ovl \F$ for all $T>0$. Further assume that $\phi^k$,
$k = 1,\ldots, m$, are functions $\F\to [0,+\infty)$ of class $\C^{2+\al}$
such that $\phi_k(0) = \lap\phi_k(x) = 0$ for $x\in\pl \F$.
Then, there exists a 
$\C^{1+\frac{\al}2,2+\al}(\ovl \F_T)$-solution $u = (u^1,\ldots, u^m)$ 
to problem \rf{lv-nd}
such that $u^k\gt 0$ for all $k = 1,\ldots, m$. This solution is unique 
in the class $\C^{1,2}(\ovl \F_T)$.
\end{thm}
\begin{proof}
Note that Assumptions \ref{A1}, \ref{A2'}, \ref{A3}--\ref{A7} are fulfilled for problem \rf{lv-nd}.
Therefore, the unique $\C^{1+\frac{\al}2,2+\al}([0,+\infty)\x \ovl \F)$-solution 
to problem  \rf{lv-nd} exists by 
Therorem \ref{th2222} and has non-negative components $u^k$.
\end{proof}
In what follows, we determine sufficient conditions when  one of the species vanishes 
as $t\to +\infty$.
To this end, we obtain two sufficient conditions for boundedness
of a certain component of the solution to \rf{lv-nd}.
\begin{lem}
Let $\beta_k$, $\gm_{ki}$, $i,k = 1,\ldots, m$,  
be functions $[0,+\infty) \x \F \to [0,+\infty)$ and let
$u = (u^1, \ldots, u^m)$ be a $\C^{1,2}([0,+\infty)\x\ovl \F)$-solution
to problem  \rf{lv-nd} such that $u^i\gt 0$ for all
$i=1,\ldots, m$. 
Assume that the function $\beta_k / \gm_{kk}$ is bounded on
$[T,+\infty)\x \F$ for some $T>0$. 
Then, the component $u^k$ is bounded on $[0,+\infty)\x \F$.
\end{lem} 
\begin{proof}
Define 
$m_k = \max\big\{\sup_{\ovl \F_T} u^k, \sup_{[T,+\infty)\x \F} 
\big(\beta_k / \gm_{kk}\big)\big\}$. Then, $\bar u^k = m _k- u^k$ satisfies
the equation
\aa{
\pl_t \bar u^k \! = d_k \lap_x \bar u^k \!+ \bar u^k\Big(\beta_k - \gm_{kk} (u^k + m_k) -
\sum_{i\ne k}^m \gm_{ki} u^i\Big)\!
+ m_k\Big(\gm_{kk}m_k-\beta_k + \sum_{i\ne k}^m \gm_{ki} u^i\Big).
}
Consider the above equation with respect to the initial moment $T$ and
note that $\bar u^k(T,x)\gt 0$.
By Theorem \ref{l1111}, $\bar u^k =  m_k - u^k \gt 0$ on
$[0,+\infty)\x \F$, which completes the proof. 
\end{proof}
\begin{lem}
Assume that $\beta_k$, $\gm_{ki}$, $i,k = 1,\ldots, m$,  
are continuous functions $[0,+\infty) \x \ovl \F \to [0,+\infty)$,
which moreover are H\"older continuous in $x\in\F$ uniformly 
with respect to $t\in [0,T]$ for all $T>0$,
 and such that $\int_0^\infty \sup_{x\in \ovl \F} \beta_k(t,x) dt <+\infty$.
Further assume that $\phi^k: \ovl \F \to [0,+\infty)$ is continuous.
Let $u = (u^1, \ldots, u^m)$ be a $\C^{1,2}([0,+\infty)\x\ovl \F)$-solution
to problem  \rf{lv-nd} such that $u^i\gt 0$ for all
$i=1,\ldots, m$. Then, the component $u^k$ is bounded on $[0,+\infty)\x \F$.
\end{lem}
\begin{proof}
By \cite[Theorems 12 and 16, Chapter 1]{friedman}, 
\aaa{
\lb{fs-1111}
u^k(t,x) = \int_{\Rnu^n} p_t(x-y) \phi^k(y) dy + \int_0^t
\int_{\Rnu^n} p_{t-s}(x-y)f_k(s,y)ds dy,
}
where 
\aa{
p_t(x) = (2\pi d_k t)^{-\frac{n}2} \exp\{-|x|^2 / (2d_k t)\} \quad
\text{and} \quad 
f_k = u_k(\beta_k - \gm_{kk} u^k - \sum_{i\ne k}^m \gm_{ki} u^i).
}
Indeed, the functions $\phi^k$ and $f_k$ vanish on $\pl \F$, and therefore, 
can be extended by zero outside of $\ovl\F$, 
so problem \rf{lv-nd} can be viewed as the Cauchy problem
\aa{
\pl_t u^k - d_k \lap_x u^k = f_k,\quad
u^k(0,x) = \phi^k(x).
}
It is worth noticing that the local H\"older continuity of $f_k$ in $x\in \Rnu^n$ 
uniformly in $t\in [0,T]$ (for all $T>0$) is a part of the requirements 
of \cite[Theorem 12, Chapter 1]{friedman}. In our situation, it
is implied by the following argument. Recall that $u^k$ is extended
by zero to $\Rnu^n$.
Take two points $x,z\in \Rnu^n$ such that $x\in\F$, $z\notin \F$
and define 
\aaa{
\lb{g1111}
g_k = \beta_k - \gm_{kk} u^k - \sum_{i\ne k}^m \gm_{ki} u^i.
}
Let $y = [x,z]\cap \pl\F$.
We have
\aa{
&f_k(t,x) - f_k(t,z) = u^k(t,x) g_k(t,x) = g_k(t,x) (u^k(t,x)-u^k(t,y)),\\
&|f_k(t,x) - f_k(t,z)| \lt  K |x-y| \lt K|x-z|, \quad 
\text{where} \;\; K = \sup_{[0,T]\x\ovl \F} |g_k \,\pl_x u^k|.
}
Note that if $x,z\in \F$ or $x,z\notin \F$, then the above inequality is obvious.
Equation  \rf{fs-1111} implies the estimate
\mm{
u^k(t,x)  \lt \int_{\F} p_t(x-y) \phi^k(y) dy 
+ \int_0^t\int_{\F} p_{t-s}(x-y)\beta_k(s,y) u^k(s,y)dy ds\\ \lt
\sup_\F|\phi^k| + \int_0^t \sup_{\F} u^k(s,\fdot) \sup_\F\beta_k(s,\fdot) ds.
}
By Gronwall's inequality,\\
\phantom{j} \hspace{1.7cm} $\sup\limits_{\F} u^k(t,x) \lt \sup\limits_\F|\phi^k|\, e^{\int_0^t \sup\limits_\F\beta_k(s,\fdot) ds}
\lt \sup\limits_\F|\phi^k| \,e^{\int_0^\infty \sup_\F\beta_k(s,\fdot) ds}.$ 
\end{proof}
In the following theorem, we assume the boundedness of the $\C^{1,2}$-solution $u$ to \rf{lv-nd}
on $[0,+\infty)\x \F$. This boundedness is implied, for instance,
by the sufficient conditions for the boundedness of separate components
obtained above.
\begin{thm}
Let $\beta_k$, $\gm_{ki}$, $i,k = 1,\ldots, m$,  
be bounded continuous functions $[0,+\infty) \x \F \to [0,+\infty)$ and let
$u = (u^1, \ldots, u^m)$ be a bounded $\C^{1,2}([0,+\infty)\x\ovl \F)$-solution
to problem  \rf{lv-nd} such that $u^i\gt 0$ for all $i=1,\ldots, m$. 
Assume
$\beta_k(t,x) \to 0$ as $t\to+\infty$ uniformly in $x\in \ovl \F$. Then,
$u^k(t,x)\to 0$ as $t\to+\infty$ uniformly in $x\in \F$
\end{thm}
\begin{proof}
Rewrite the equation for $u^k$ as follows:
\aa{
\pl_t u^k = d_k \lap_x u^k + u^k g_k,
}
where $g_k$ is defined by \rf{g1111}. Let us observe that 
$\limsup_{t\to+\infty} g_k(t,x) \lt 0$ uniformly in $x\in \ovl \F$
and that $g_k$ is bounded on $[0,+\infty) \x \ovl \F$. 
The conclusion then holds by \cite[Theorem 1, Chapter 6]{friedman}.
\end{proof}

\subsection{Asymptotic behavior for the 2-species 
Lotka-\!Volterra competition model}
In this subsection, we prove Theorem \ref{th-lv2222}
that describes the asymptotic behavior
of the solution $(u,v)$ to system \rf{lv1111}.
\begin{proof}[Proof of Theorem \ref{th-lv2222}]
 It follows from the assumptions that the functions $f_1(t,x) =u(\beta- \gm u - \dl v)(t,x)$ and 
$f_2(t,x) =v(\rho - \sg u - \te v)(t,x)$ are of class 
$\C^{\frac\al2,\al}(\ovl \F_T)$ for all $T>0$.
By  \cite[Theorem 10, Chapter 3]{friedman}, the solution $(u,v)$ and its laplacian
$(\lap_x u, \lap_x v)$ are differentiable in $t$ and the derivatives $\pl_t \lap_x u$ and
$\pl_t \lap_x v$ are continuous. Therefore
$\hat u = \pl_t u$ and $\hat v =  -\pl_t v$ solve the system 
 \eqq{
\pl_t \hat u = d_1 \lap_x \hat u + \hat u(\beta- 2\gm u - \dl v) + \dl u \hat v
+ u(\pl_t \beta - \pl_t \gm \, u - \pl_t \dl\,  v),\\
\pl_t \hat v = d_2 \lap_x \hat v + \hat v(\rho - \sg u - 2\te v) + \sg v \hat u
+ v( \pl_t \sg\, u + \pl_t \te\, v - \pl_t \rho).
}
By Theorem \ref{l1111}, $\pl_t u = \hat u \gt 0$ 
and $\pl_t v = - \hat v \lt 0$ on $[0,+\infty)\x \F$.
Therefore, for each $x\in \F$, the function $t\mto u(t,x)$ is non-decreasing,
and the function $t\mto v(t,x)$ is non-increasing. Since $(u,v)$ is a bounded
solution, the pointwise limits \rf{limits} exist.
Let us show now that $(\bar u,\bar v)$ is a weak solution to \rf{el-1111}.
Take two test functions $\eta(x)$ and $\zeta(x)$ 
such that that ${\rm  supp\;} \eta \sub \F$, ${\rm  supp\;} \zeta \sub \F$.
Furthermore, take a test function $\tet(t)$ of variable $t$ compactly supported on 
$(0,+\infty)$ and such that  $\int_0^{\infty} \tet(t) dt \ne 0$.
Let us show that we can pass to the limit as $t\to +\infty$ in the equation
for $u(t,x)$. 

For $\tau>0$, introduce the functions 
$u_\tau(t,x) = u(\tau+t,x)$ and $v_\tau(t,x)= v(\tau+t,x)$ 
and note that $(u_\tau,v_\tau)$ solves the equation
\aa{
\pl_t u_\tau = d_1 \lap_x u_\tau + u_\tau(\beta_\tau- \gm_\tau u_\tau - \dl_\tau v_\tau),
}
where the functions $\beta_\tau,\gm_\tau,\dl_\tau$
are defined
in the similar manner as $u_\tau$; that is, $\beta_\tau(t,x) = \beta(t+\tau,x)$, etc.
Multiplying the equation for $u_\tau$  by $\eta(x)$, integrating over $\F$, 
and next, multiplying the resulting equation
by $\tet(t)$ and integrating over $[0,+\infty)$, by Fubini's theorem,
we obtain 
\aa{
-\int_\F \eta \int_0^{\infty} u_\tau \,\pl_t \tet dt \,  dx = \int_0^{\infty} \tet \int_\F
\big(d_1  \lap_x \eta \,u_\tau + \eta\, u_\tau(\beta_\tau- \gm_\tau u_\tau 
- \dl_\tau v_\tau)\big) dx \, dt.
}
By boundedness of the solution $(u,v)$ and since
$0\lt \beta_\tau \lt \bar \beta$, $\bar \gm \lt \gm_\tau \lt \gm(0,\fdot)$,
$\bar \dl \lt \dl_\tau \lt \dl(0,\fdot)$, we can pass to the limit under the
integral sign as $\tau\to+\infty$. We obtain 
\aa{
0 = - \int_\F \eta \,\bar u\, dx \, \big(\left.\tet\right|^{\infty}_0\big) 
=\int_0^\infty \tet dt \int_\F \big(d_1 \lap_x \eta\, \bar u  
+ \eta\, \bar u (\bar \beta- \bar \gm\bar u - \bar \dl\bar v) \big) dx.
}
Since, $\int_0^\infty \tet dt \ne 0$, $(\bar u,\bar v)$ solves 
the first equation in \rf{weak1111}.
The equation for $v$ is analyzed likewise.
The theorem is proved. 
\end{proof}
\begin{rem}
\rm
Note that unlike the previous subsection, here we do not use the existence
of a boundary function 
\eqq{
u(0,x) = \ffi(x), \quad v(0,x) = \psi(x), \quad x\in \F, \\
0, \quad (t,x) \in  [0,+\infty) \x \pl \F.
}
Also remark that if the initial condition is given, the assumption
$\pl_t u(0,x)\gt 0$, $\pl_t v(0,x)\lt 0$ can be verified as follows:
$d_1 \lap_x \ffi + \ffi(\beta- \gm \ffi - \dl \psi) \gt 0$ and
 $d_2 \lap_x \psi + \psi(\rho - \sg \ffi - \te \psi) \lt 0$.
 \end{rem}

\end{document}